\documentclass[11pt]{amsart}

\usepackage{amsmath}
\usepackage{amssymb}
\usepackage{amsthm}

\theoremstyle{definition}
\newtheorem{definition}{Definition}[section]
\newtheorem{example}[definition]{Example}
\newtheorem{proposition}[definition]{Proposition}
\newtheorem{theorem}[definition]{Theorem}
\newtheorem{lemma}[definition]{Lemma}

\newtheorem{remark}[definition]{Remark}

\usepackage{enumitem}
\usepackage{varioref} 

\usepackage[margin=2.0cm]{geometry}

\def\R{\mathbb R}

\newcommand{\inner}[2]{\langle #1,#2\rangle}
\newcommand{\innerbig}[2]{\left\langle #1,#2\right\rangle}
\newcommand{\norm}[1]{\|#1\|}

\DeclareMathOperator{\LspaceSymbol}{\mathbf{L}} 
\DeclareMathOperator{\WspaceSymbol}{\mathbf{W}} 

\newcommand{\Lspace}[1]{\LspaceSymbol^{#1}}
\newcommand{\Wspace}[1]{\WspaceSymbol^{#1}}
\newcommand{\Wzspace}[1]{\WspaceSymbol^{#1}_0}

\newcommand{\Lpspace}[1][p]{\Lspace{{#1}}}
\newcommand{\LGspace}[1][G]{\Lspace{{#1}}}
\newcommand{\WLGspace}[1][G]{{\Wspace{1}}\Lspace{{#1}}}
\newcommand{\WzLGspace}[1][G]{{\Wzspace{1}}\Lspace{{#1}}}

\newcommand{\Lpnorm}[2][p]{\norm{#2}_{\Lpspace[{#1}]}}
\newcommand{\LGnorm}[2][G]{\norm{#2}_{\LGspace[{#1}]}}
\newcommand{\LGastnorm}[2][{G^\ast}]{\norm{#2}_{\LGspace[{#1}]}}
\newcommand{\WLGnorm}[2][G]{\norm{#2}_{\WLGspace[{#1}]}}
\newcommand{\WzLGnorm}[2][G]{\norm{#2}_{\WzLGspace[{#1}]}}

\newcommand{\WLGnormother}[2][G]{\norm{#2}_{1,\WLGspace[{#1}]}}

\DeclareMathOperator{\ISymbol}{\mathcal{I}}

\title[Anisotropic Orlicz-Sobolev spaces of vector valued functions ]{
Anisotropic Orlicz-Sobolev spaces of vector valued functions and Lagrange equations
}

\author{M. Chmara}

\author{J. Maksymiuk}

\address{Department of Technical Physics and Applied Mathematics, Gda\'{n}sk University of
Technology, Narutowicza 11/12, 80-952 Gda\'{n}sk, Poland}
\email{mchmara@mif.pg.gda.pl, jmaksymiuk@mif.pg.gda.pl}

\begin{document}
\begin{abstract}
In this paper we study some properties of anisotropic Orlicz and anisotropic Orlicz-Sobolev spaces
of vector valued functions for a special class of G-functions. We introduce a variational setting
for a class of Lagrangian Systems. We give conditions which ensure that the principal part of
variational functional is finitely defined and continuously differentiable on Orlicz-Sobolev space.
\end{abstract}


\maketitle


\section{Introduction}

In this paper we make some preliminary steps for variational analysis in anisotropic
Orlicz-Sobolev spaces of vector valued functions. We consider the Euler-Lagrange equation
\begin{equation}
    \label{eq:L}
    \frac{d}{dt}L_v(t,u(t),\dot{u}(t))=L_x(t,u(t),\dot{u}(t)),\quad t\in (a,b)
\end{equation}
where Lagrangian is of the form $L(t,x,v)=F(t,x,v)+V(t,x)$.

If $F(v)=\frac12|v|^2$ then the equation \eqref{eq:L} reduces to $\ddot{u}(t)+\nabla V(t,u(t))=0$.
One can
consider more general case $F(v)=\phi(|v|)$, where $\phi$ is convex and nonnegative.  In the above
cases $F$ does not depend on $v$ directly but rather on its norm $|v|$ and the growth of $F$ is the
same in all directions, i.e. $F$ has isotropic growth. Equation \eqref{eq:L} with Lagrangian
$L(t,x,v)=\frac{1}{p}|v|^p+V(t,x)$ has been studied by many authors under different conditions. The
classical reference is \cite{MawWil89}. The isotropic Orlicz-Sobolev space setting was considered in
\cite{AciBurGiuMazSch15}.

We are interested in anisotropic case. This means that $F$ depends on all components of $v$ not only
on $|v|$ and has different growth in different directions. A simple example of such function is
$F(v)=\sum_{i=1}^N |v_i|^{p_i}$ or $F(v)=\sum_{i=1}^N \phi_i(|v_i|)$, where $\phi_i$ are
N-functions. We wish to consider more general situation. We assume that $F\colon
[a,b]\times\R^N\times\R^N\to \R$ satisfies
 \begin{enumerate}[label=($F_\arabic*$)]
    \item $F\in C^1$,
    \item $ |F(t,x,v)|\leq a(|x|)(b(t)+G(v)), $
    \item $|F_x(t,x,v)|\leq a(|x|)(b(t)+G(v)),$
    \item $G^{\ast}(F_v(t,x,v))\leq a(|x|)(c(t)+G^{\ast}(\nabla G(v))).$
\end{enumerate}
where $a\in C(\R_+,\R_+)$, $b,c\in\LGspace[1](I,\R_+)$ and $G\colon \R^N\to \R$ is a G-function.
Conditions $(F_1)$--$(F_4)$ are direct generalization of standard growth conditions from
\cite{MawWil89} (see also \cite{AciBurGiuMazSch15}).
We show (see Theorem \ref{thm:IFabcC1}) that under these conditions the functional
$\ISymbol\colon \WLGspace\to \R$ given by
\[
\ISymbol(u)=\int_I F(t,u,\dot{u})\,dt
\]
is continuously differentiable.

We restrict our considerations to a special class of G-functions. Here $G\colon\R^n\to [0,\infty)$
is convex, $G(-x)=G(x)$, supercoercive, $G(0)=0$ and satisfies $\Delta_2$ and $\nabla_2$
conditions. We define the anisotropic Orlicz space to be
\[
    \LGspace(I,\R^N)=\{u\colon I\to \R^N\colon \int_I G(u)\,dt\leq \infty\}.
\]
The Orlicz space $\LGspace$ equipped with the Luxemburg norm
\[
\LGnorm{u}=\inf\left\{\alpha>0\colon \int_I G\left(\frac{u}{\alpha}\right)\, dt\leq 1 \right\}.
\]
is a reflexive Banach space. An important example of Orlicz space is classical Lebesgue $\Lpspace$
space, defined by $G(x)=\frac1p|x|^p$. In this case, the Luxemburg norm and the standard $\Lpspace$
norm are equivalent. Therefore, Orlicz spaces can be viewed as a straightforward generalization of
$\Lpspace$ spaces.

Properties of N-functions and of Orlicz spaces of real-valued functions has been studied in great
details in monographs \cite{KraRut61,RaoRen91,RaoRen02} and \cite{Ada75}. The standard references
for vector-valued case are \cite{Ska69_I, Ska69_II, Tru74} and \cite{DesGri01, Sch05} for
Banach-space valued functions. In \cite{Ska69_I,Ska69_II} author considers a class of G-functions
together with a uniformity conditions which, for example excludes the function $G(x)=\sum
|x_i|^{p_i}$ unless $1<p_1=\dots=p_N<\infty$. Moreover $G$ is not neccessairly assumed to be an even
function. As was pointed out in \cite{Sch05}, if $G$ is not even then $\LGspace$ is no longer a
vector space (see also \cite[Example 2.1]{DesGri01}).

Our strong conditions on $G$ allow us to work in Orlicz spaces without worry about some
technical difficulties arising in general case. For example,  it is well known that the set
$\LGspace(I,\R^N)$ is a vector space if and only if $G$ satisfies $\Delta_2$ condition. Otherwise
$\LGspace$ is only a convex set. Another difficulty is the convergence notion. In Lebesgue spaces
$\Lpnorm{u_n-u}\to 0$ means simply $\int |u_n-u|^p\to 0$. For arbitrary G-function $G$, convergence
in Luxemburg norm is not equivalent to $\int G(u_n-u)\,dt\to 0$ unless $G$ satisfies $\Delta_2$. The
$\Delta_2$ condition is also crucial for separability and reflexivity of $\LGspace$.

The main consequence of anisotropic nature of $G$ is the lack of monotonicity of the norm. It is no
longer true that $|u|\leq |v|$ implies $\LGnorm{u}\leq \LGnorm{v}$. In anisotropic case, standard
dominance condition $|u_n|\leq f$ does not implies convergence in $\LGspace$ norm and must be
replaced by $G(u_n)\leq f$ (see Theorem \ref{thm:dominatedconvergence}).

Following \cite{DesGri01} we show that for every $G$ we consider there exist $p,q\in(1,\infty)$ such
that $\LGspace[q]\hookrightarrow\LGspace[G]\hookrightarrow\LGspace[q]$. If $G(x)=\sum |x_i|^{p_i}$
then $\LGspace$ can be identified with the product of $\Lpspace[p_i]$ but in many cases an
anisotropic Orlicz Space is not equal to the space $\Lpspace[p_1]\times
\Lpspace[p_2]\times...\times \Lpspace[p_N]$  (see Example \ref{ex:LGnotLPtimesLP}).

To give a proper variational setting for equation \eqref{eq:L} we introduce a notion of an
anisotropic Orlicz-Sobolev space $\WLGspace$ of vector-valued functions. It is defined to be
\[
    \WLGspace(I,\R^N)=\{u\in \LGspace(I,\R^N)\colon \dot{u}\in \LGspace(I,\R^N)\}
\]
with the norm
\[
    \WLGnorm{u}=\LGnorm{u}+\LGnorm{\dot{u}}
\]
To the authors best knowledge there is no reference for the case of anisotropic norm and
vector-valued functions of one variable. The references for other cases are
\cite{AciBurGiuMazSch15,Tru74, FucOsm98, Cianchi00,Cianchi96,Cianchi99, Cletal04, DonTru71,
JaiLuk02, Le14}.

In  \cite{Tru74} and \cite{JaiLuk02} the  space $H^0(G,\Omega)$, $\Omega\subset\R^n$ is defined
as a completion of $C_0^1(\Omega,\R^n)$ under norm $\norm{u}_{H^0(G,\Omega)}=\norm{Du}_{G,\Omega}$.
It is classical result due to Trudinger $H^0(G,\Omega)\hookrightarrow L_A(\Omega)$, where $A$ is
some N-function (see also Cianchi \cite{Cianchi96}).

In \cite{DonTru71} and \cite{Le14} the anisotropic Orlicz-Sobolev space $W^1L_G$ is defined for
G-function $G:\R^{n+1}\to[0,\infty]$  as a space of weakly differentiable functions
$u:\R^n\supset \Omega\to\R$ such that $(u,D_1u,D_2u, ...,D_nu)$ belongs to the Orlicz space
generated by $G$. A norm for  $W^1L_G$ is given by
\[
\norm{u}_{1,G,\Omega}=\norm{(u,Du)}_{G,\Omega}.
\]
In \cite{FucOsm98} we can find definition of isotropic Orlicz-Sobolev space of real valued
functions
\[
W^1_A(\Omega)=\{u\in\Omega\to\R \text{ measurable }: u, |\nabla u|\in L_A \},
\]
where $L_A$ is Orlicz Space and $A$ is an N-function.

In \cite{AciBurGiuMazSch15} the isotropic Orlicz-Sobolev space if vector-valued functions is
defined to be a space of absolutely continuous functions $u:[0,T]\to\R^d$ such that $u$ and
$\dot{u}$ belongs to Orlicz space generated by an N-function. Similar treatment can be found in
\cite{RadRep15}.

\section{G-functions}
\label{sec:G}

Let $\inner{\cdot}{\cdot}$ denote the standard inner product on  $\R^N$ and $|\cdot|$ is the
induced norm.  We assume that $G\colon\R^N\to [0,\infty)$ satisfies the following conditions:
\begin{enumerate}[label=($G_\arabic*$),series=G]
    \item \label{G:0in0}
    $G(0)=0$,
    \item \label{G:convex}
    $G$ is convex,
    \item \label{G:symmetric}
    $G$ is even,
    \item \label{G:supercoercivity}
    $G$ is supercoercive:
    \[
    \lim_{|x|\to \infty}\frac{G(x)}{|x|}=\infty,
    \]
    \item \label{G:delta2}
    $G$ satisfies the $\Delta_2$ condition:
    \begin{equation}\tag{$\Delta_2$}
        \exists_{K_1\geq 2}\ \exists_{M_1>0}\ \forall_{|x|\geq M_1}\ G(2x)\leq K_1G(x),
    \end{equation}
    \item \label{G:nabla2}
    $G$ satisfies the $\nabla_2$ condition:
    \begin{equation}\tag{$\nabla_2$}
        \exists_{K_2\geq 1}\ \exists_{M_2>0}\ \forall_{|x|\geq M_2}\ G(x)\leq \frac{1}{2K_2}G(K_2x).
    \end{equation}
\end{enumerate}

A function $G$ is a G-function in the sense of Trudinger \cite{Tru74}. In general, G-function can
be unbounded on bounded sets and need not satisfy conditions
\ref{G:supercoercivity}--\ref{G:nabla2} but only $\lim_{x\to\infty} G(x)=\infty$. A G-function of
one variable is called N-function. Some typical examples of $G$ are
\begin{enumerate}
  \item $G(x)=\frac1p|x|^p$, $1<p<\infty$
  \item $G(x)=\sum_{i=1}^N G_{p_i}(x_i)$, $1<p_i<\infty$
  \item $G(x)=(x_1-x_2)^{2}+x_2^{4}$
\end{enumerate}
A function $G$ can be equal to zero in some neighborhood of $0$. So that a function
\[
G(x)=\begin{cases}0&|x|\leq 1\\ |x|^2-1&|x|>1\end{cases}
\]
is also admissible. Conditions $\Delta_2$ and $\nabla_2$ implies that $G$ is of polynomial growth
(see Lemma \ref{lemma:xp<G<xq} below and \cite{KraRut61}). A function $f:\R^2\to\R$
$f(x)=e^{|x|}-|x|-1$ does not satisfy $\Delta_2$.

Since $G$ is convex and finite on $\R^n$, $G$ is locally Lipschitz and therefore continuous. Note
that for every $x\in \R^N$
\begin{equation*}
\label{G:properties:homogenity}
    \begin{aligned}
        G(\alpha x)\leq \alpha G(x),& \text{ if } 0\leq \alpha \leq 1,
        \\
        \alpha G(x)\leq G(\alpha x),& \text{ if } 1 \leq \alpha.
    \end{aligned}
\end{equation*}
We obtain immediately that $G$ is non-decreasing along any half-line through the origin i.e. for
every $x\in \R^N$
\begin{equation}
\label{G:properties:nondecreasing}
    0<\alpha\leq \beta \implies G(\alpha x)\leq G(\beta x).
\end{equation}
Our assumptions on $G$ imply that for every $x_0\in \R^N$ there exists $a\in \R^N$ and $b\in \R$
such that for all $x\in\R^N$
\[
\inner{a}{x_0}+b=G(x_0) \text{ and } \inner{a}{x}+b\leq G(x).
\]
From this, we can easily obtain the Jensen integral inequality. Let $I\subset \R$ be a finite
interval and let $u\in \Lpspace[1](I,\R^N)$. Then
\[
G\left(\frac{1}{\mu(I)}\int_I u\,dt\right)\leq \frac{1}{\mu(I)}\int_I G(u)\,dt.
\]

We will often make use of the following simple observation.

\begin{proposition}
    \label{lemma:delat2charact}
    For all $\alpha\in \R$ there exists $K_1(\alpha)> 0$ such that
    \[
    G(\alpha x)\leq K_1(\alpha)G(x)
    \]
    for all $|x|\geq M_1$.
\end{proposition}

In fact, the above proposition provides a characterization of $\Delta_2$ (see \cite{Ska69_I,Sch05}).
It follows that for every $\alpha\in \R$ there exists $C_\alpha>0$ such that for $x\in \R^N$
\[
  G(\alpha x)\leq C_\alpha+K_1(\alpha)G(x).
\]

We recall a notion of Fenchel conjugate. Define $G^{\ast}:\R^N\to[0,\infty)$ by
\begin{equation*}
    \label{eq:conjugate_def}
  G^{\ast}(y):=\sup_{x\in\R^N}\{\inner{x}{y}-G(x)\}.
\end{equation*}
A function $G^\ast$ is called Fenchel conjugate of $G$. As an immediate consequence of definition we
have  the so called Fenschel inequality:
\begin{equation*}
\label{ineq:Fenshel}
    \forall_{x,y\in\R^N}\ \inner{x}{y}\leq G(x)+G^{\ast}(y).
\end{equation*}

Consider arbitrary $f\colon \R^N\to [0,\infty)$. It is obvious that the conjugate function
$f^\ast$ is always convex. But in general $f^\ast$ need not be continuous, finite or coercive, even
if $f$ is. From the other hand, it is well known that if $f$ is convex and l.s.c. then
$f^\ast\not \equiv \infty$ and $(f^\ast)^\ast=f$.
\begin{example}

    \begin{enumerate}
            \item
        If
         \[
        g(x)=\begin{cases}
        0&|x|\leq 1\\
        \infty&|x|>1
        \end{cases}
        \]
    then   $g^{\ast}(x)=|x|.$
Note that $g$ and $g^{\ast}$ are G-functions but do not satisfy our assumptions.

\item If $G(x)=\frac{1}{p}|x|^p$, then $G^\ast(x) = \frac1q|x|^q$, $\frac{1}{p}+\frac{1}{q}=1$.
\item If $G(x)=\sum_{i=1}^N G_i(x_i)$ then $G^\ast(x)=\sum_{i=1}^N G_i^\ast(x_i)$.
\item If $G(x,y)=(x-y)^2+y^4$,  then
\[
G^{\ast}(x,y)=\frac14x^2+\frac{3}{4}(x+y)\left(\frac{x+y}{4}\right)^{\frac13}.
\]
    \end{enumerate}
\end{example}

More information on general theory of conjugate functions can be found in standard books on convex
analysis, see for instance \cite{HirLem04,Aub98}.

If a function $G\colon \R^n\to [0,\infty)$ satisfies conditions \ref{G:0in0}--\ref{G:nabla2} then
the same is true for its conjugate $G^\ast$. This is main reason we want to restrict class of
considered functions.

\begin{theorem}
    If $G$ satisfies conditions \ref{G:0in0}--\ref{G:nabla2} then $G^\ast$ also satisfies
    \ref{G:0in0}--\ref{G:nabla2} and $(G^\ast)^\ast=G$.
\end{theorem}
\begin{proof}
    It is evident that $G^\ast$ satisfies \ref{G:0in0}, \ref{G:convex} and \ref{G:symmetric}.
    It is well known that under our conditions,
    $G^\ast$ is finite (proposition 1.3.8, \cite{HirLem04}),
    $G^\ast$ is supercoercive (proposition 1.3.9, \cite{HirLem04}) and
    $G^\ast$ satisfies \ref{G:delta2} and \ref{G:nabla2} (remark 2.3, \cite{DesGri01}).
    Corrollary \cite[cor. 1.3.6]{HirLem04} gives $(G^\ast)^\ast=G$.
\end{proof}

In order to compare growth rate of G-functions we define two relations. Let $G_1$ and $G_2$ be
G-functions. Define
\[
    G_1\prec G_2
    \iff
    \exists_{M\geq 0}\ \exists_{K>0}\ \forall_{|x|\geq M}\ G_1(x)\leq G_2(K\,x)
\]
and
\[
    G_1\prec\prec G_2
    \iff
    \forall_{\alpha>0}\ \lim_{|x|\to \infty}\frac{G_2(\alpha x)}{G_1(x)}=\infty.
\]
For conjugate functions we have (see \cite[thm. 3.1]{KraRut61})
\begin{equation*}
    \label{prop:F<GtoG*<F*}
    G_1\prec G_2 \Rightarrow G_2^{\ast}\prec G_1^{\ast}.
\end{equation*}

Obviously $G_1\prec\prec G_2$ implies $G_1\prec G_2$. Assumption \ref{G:supercoercivity} implies
$|x|\prec\prec G$. It is true that $|x|\prec G$ holds under weaker assumption: $G(x)\to \infty$.
Note that, if $p>1$ then $|x|\prec\prec |x|^p$. Hence, if $|x|^p\prec G$ then $|x|\prec\prec
G$. Since $G$ satisfies \ref{G:delta2} and \ref{G:nabla2} we have the following bounds for the
growth of $G$.

\begin{lemma}[{cf. \cite[Lemma 2.4]{DesGri01}}]
    \label{lemma:xp<G<xq}
There exists $p,q\in (1,\infty)$ such that
\[
    |x|^p\prec G\prec |x|^q.
\]
\end{lemma}
\begin{proof}
Set $C=\overline{G}(M_1)$. By induction, if $|x|\leq 2^n M_1$ then $G(x)\leq K_1^n C$. For $|x|\geq
M_1$ choose $n$ such that $2^{n-1}M_1\leq |x|\leq 2^nM_1$. Then $n-1\leq \log_2(|x|/M_1)$ and
$G(x)\leq C K_1^{1+\log_2(|x|/M_1)}$. Therefore, for $|x|\geq M_1$,
\[
G(x)\leq CK_1M_1^{-q}\,|x|^q,\quad q=\log_2(K_1).
\]
This proves that $G\prec |x|^q$. Choose $r>0$ such that if $x\in G^{-1}(\overline{G}(M_1))$ then
$|x|\leq r$. Set $M=rM_1$. Again, by induction, for $|x|\geq K_2^k M$ we have $(2K_2)^kC\leq G(x)$.
This implies
\[
G(x)\geq C(2K_2)^{-q}|x|^{q},\quad p=1+\frac{1}{\log_2 (K_2)}
\]
whenever $|x|\geq MK_2$. Hence $|x|^p\prec G$.
\end{proof}

Immediately from the above we get $|x|^{\frac{q}{q-1}}\prec G^\ast \prec |x|^{\frac{p}{p-1}}$.

\section{Orlicz spaces}
\label{sec:Ospace}

Let $I\subset \R$ be a finite interval. The Orlicz space $\LGspace=\LGspace(I,\R^n)$ is defined to
be
\[
\LGspace(I,\R^n)=
\left\{
u\colon I\to \R^n\colon u\text{ - measurable }
\int_I G\left(u\right)\,dt<\infty
\right\}.
\]

As usual, we identify functions equal a.e. For an arbitrary G-function $f\colon \R^n\to [0,\infty)$
which does not satisfies $\Delta_2$ the set $\LGspace[f]$ is not a linear space but only a convex
set. In fact, it is well known that the set $\LGspace[f]$ is linear space if and only if a
G-function $f$ satisfies $\Delta_2$ condition.

For $u\in \LGspace$ define:
\[
\LGnorm{u}=\inf\left\{\alpha>0\colon \int_I G\left(\frac{u}{\alpha}\right)\, dt\leq 1
\right\}.
\]
The function $\LGnorm{\cdot}$ is called the Luxemburg norm. It is easy to see that
\[
    \int_IG\left(\frac{u}{\LGnorm{u}}\right)\, dt = 1,
\]
since $G$ satisfies $\Delta_2$. Moreover
\[
  \int_IG\left(\frac{u}{k}\right)\, dt\leq 1\iff \LGnorm{u}\leq k.
\]

\begin{remark}
All properties of $\LGspace$ remains true for $\LGspace[{G^\ast}],$ since $G$ and $G^\ast$
belongs to the same class of functions.
\end{remark}

\begin{theorem}\label{th:Luxemburgisnorm}
    If $G\colon \R^n\to [0,\infty)$ satisfies \ref{G:0in0}--\ref{G:nabla2}, then
$(\LGspace(I,\R^n),\LGnorm{\cdot})$ is a normed linear space.
\end{theorem}
\begin{proof}
We first prove that $\LGspace$ is a linear space.  Since $G$ is continuous and satisfies $\Delta_2$,
we get
\[
\int_ I G(\alpha u)\, dt=
\int_{ I_1} G(\alpha u)\, dt+\int_{ I\setminus I_1} G(\alpha u)\, dt
\leq
\mu(I_1)\,C_\alpha +K_1(\alpha) \int_ I G(u)\, dt <\infty
\]
where $ I_1={\{t\in I\colon |u(t)|\leq M_1\}}$. Hence, if $u\in \LGspace$ then $\alpha
u\in\LGspace$ for all $\alpha\in \R$. For every $u,v\in \LGspace$ and $\alpha,\beta\in \R$, by
\ref{G:convex} and Proposition \ref{lemma:delat2charact}, we have
\[
\int_ I G(\alpha u+\beta v)dt
\leq
\frac12 \int_ I G(2\alpha u)dt+\frac12 \int_ I G(2\beta v)dt<\infty.
\]
Hence $\alpha u+\beta v\in \LGspace$.

Now we show that $\LGnorm{\cdot}$ is a norm on $\LGspace$.
It is evident that if $u=0$ then $\LGnorm{u}=0$. Suppose $u\neq 0$. There exists $I_1\subset I $
with positive measure and $\varepsilon>0$ such that for all $t\in I_1$, $|u(t)|\geq \varepsilon$.
For every $t\in I_1$ there exists $\alpha_t\geq 1$ and $y_t\in \R^n$, $|y_t|=\varepsilon$ such that
$u(t)=\alpha_t y_t$. For all $k>0$ we have by \eqref{G:properties:nondecreasing} that $G(\alpha_t
y_t/k)\geq G(y_t/k)$. Hence
\[
    \int_I  G\left(\frac{u}{k}\right)\,dt
    \geq
    \int_{I _1} G\left(\frac{u}{k}\right)\,dt
    =
    \int_{I _1} G\left(\frac{\alpha_t y_t}{k}\right)\,dt
    \geq
    \int_{I _1} G\left(\frac{y_t}{k}\right)\,dt
    \geq
    \mu(I_1) \underline G(\varepsilon/k),
\]
where $\underline G(\varepsilon/k) = \inf\{G(y)\colon |y|=\frac{\varepsilon}{k}\}$.  Since
$\underline G(\varepsilon/k)\nearrow \infty$ as $k\searrow 0$, there exists $k_0>0$ such that for
all
$k\leq k_0$
\[
    \int_I  G\left(\frac{u}{k}\right)dt > 1
    \]
    and
    \[
    \LGnorm{u}=\inf\left\{k>0: \int_I  G\left(\frac{u}{k}\right)dt\leq 1\right\}\geq k_0>0.
    \]
    Finally, $\LGnorm{u}=0\iff u=0$. Let $u\in L^G$ and $\alpha\in \R$.  For $\alpha\in \R$:
\begin{multline*}
    \LGnorm{\alpha u}
    =
    \inf\left\{k>0: \int_I  G\left(\frac{\alpha u}{k}\right) \leq 1\right\}
    =
    |\alpha| \inf\left\{k/|\alpha|>0: \int_I  G\left(\frac{ u}{k/|\alpha|}\right)\leq
    1\right\}
    =
    |\alpha|\LGnorm{u}.
\end{multline*}

    If  $\LGnorm{u}=0$ or $\LGnorm{v}=0$, then it is obvious that
    $\LGnorm{u+v}\leq\LGnorm{u}+\LGnorm{v}$. Set $\alpha=\LGnorm{u}>0$, $\beta=\LGnorm{v}>0$. Then
    $\int_I  G\left(\frac{ u}{\alpha}\right)=1$ and
    $\int_I  G\left(\frac{v}{\beta}\right)=1$. Thus
    \begin{equation*}
        \int_I  G\left(\frac{u+v}{\alpha+\beta}\right)dt
        \leq
        \frac{\alpha}{\alpha+\beta}\int_I  G\left(\frac{u}{\alpha}\right)dt
        +
        \frac{\beta}{\alpha+\beta}\int_I  G\left(\frac{ v}{\beta}\right)dt=1.
    \end{equation*}
    As a consequence
    \begin{equation*}
        \int_I  G\left(\frac{ u+v}{\LGnorm{u}+\LGnorm{v}}\right)dt\leq 1
        \implies \LGnorm{u+v}\leq\LGnorm{u}+\LGnorm{v}.
    \end{equation*}
\end{proof}

An important example of Orlicz space is a classical Lebesgue space $(\Lpspace,\Lpnorm{\cdot})$,
$p\in(1,\infty)$ defined by $G(x)=\frac1p|x|^p$. It is easy to check that in this case
$\LGspace=\Lpspace$ and the Luxemburg norm and standard $\Lpspace$ norm are equivalent. Two
important examples of Lebesgue spaces are not covered in our setting, namely $\Lpspace[1]$ and
$\Lpspace[\infty]$. The space $\Lpspace[1]$ is generated by $f(x)=|x|$ and the space
$\LGspace[\infty]$ generated by $f^\ast$. We exclude these two spaces because we want to have only
reflexive spaces in the class of Orlicz spaces we consider.

It was pointed out by Schappacher \cite[example 3.1]{Sch05} that if $f$ is not bounded on bounded
sets (i.e. we allow $f(x)=+\infty$ for some $x\in \R^n$) then $\LGspace[f]$ need not  be a linear
space, even if $f$ satisfies $\Delta_2$ condition. To see this, consider
\[
f(x)=\begin{cases}
         \frac{1}{1-|x|}-1&|x|<1\\
         \infty& |x|\geq 1
     \end{cases}
     \text{ and } u(t)=t/2
\]
See \cite{KraRut61,Sch05} for more details.

\begin{theorem}[H\"older inequality]
    For every $u\in \LGspace$ and $v\in \LGspace[G^\ast]$
    \[
        \int_I \inner{u}{v}\,dt\leq 2\LGnorm{u}\LGnorm[G^\ast]{v}
    \]
\end{theorem}
\begin{proof}
    Using Fenchel inequality we obtain
    \[
    \innerbig{\frac{u}{\LGnorm{u}}}{\frac{v}{\LGnorm[G^\ast]{v}}}\leq
    G\Big(\frac{u}{\LGnorm{u}}\Big)+G^\ast\Big(\frac{v}{\LGnorm[G^\ast]{v}}\Big).
    \]
Hence
\[
    \int_I\innerbig{\frac{u}{\LGnorm{u}}}{\frac{v}{\LGnorm[G^\ast]{v}}}\,dt
    \leq
    \int_IG\Big(\frac{u}{\LGnorm{u}}\Big)\,dt+\int_I
G^\ast\Big(\frac{v}{\LGnorm[G^\ast]{v}}\Big)\,dt
    \leq 2.
\]
\end{proof}

We finish this section by completeness of Orlicz space.

\begin{theorem}[{cf. \cite{KraRut61}, \cite[theorem 6.1]{Sch05}}]
    The space $(\LGspace,\LGnorm{\cdot})$ is complete.
\end{theorem}
\begin{proof}
Let $\{u_n\}$ be a Cauchy sequence in $\LGspace$. Fix $\delta,\epsilon>0$ and choose $\alpha>0$ such
that $G(\alpha x)>2/\delta$ if $|x|\geq \varepsilon$. Let $n_0$ be large enough so that
$\LGnorm{u_n-u_m}\leq \alpha^{-1}$, i.e
\[
\int_I G(\alpha (u_n-u_m))\,dt\leq 1.
\]
Put
$
E=\{t\colon G(\alpha (u_n(t)-u_m(t) ))>\delta/2\}.
$
Then
\[
\frac{2}{\delta} \mu(E) \leq \int_I G(\alpha (u_n-u_m))\,dt\leq 1
\]
that is $\mu(E)<\frac{\delta}{2}$. It follows that
\[
\mu\left(\{t\colon|u_n(t)-u_m(t)|\geq \varepsilon\}\right)\leq \frac{\delta}{2}
\]
Thus $\{u_n\}$ is a Cauchy sequence in measure. This follows that there is a subsequence
$\{u_{n_k}\}$ convergent a.e. to some measurable function $u$.

Fix $\varepsilon>0$ and choose $K$ such that for $k,l>K$, $\LGnorm{u_{n_k}-u_{n_l}}\leq \epsilon$.
Then
\[
\int_I G\left(\frac{u_{n_k}-u_{n_l}}{\varepsilon}\right)\, dt
\leq
\int_I G\left(\frac{u_{n_k}-u_{n_l}}{\LGnorm{u_{n_k}-u_{n_l}}}\right)\, dt = 1.
\]
Letting $n_l\to \infty$ we obtain by Fatou Lemma,
\[
\int_I G\left(\frac{u_{n_k}-u}{\varepsilon}\right)\, dt \leq 1.
\]
Hence $u_{n_k}-u\in \LGspace$ and consequently $u\in \LGspace$. Since $\varepsilon>0$ is arbitrary,
$\LGnorm{u_{n_k}~-~u}\to 0$ and $\LGnorm{u_n-u} \to 0$.
\end{proof}

\subsection{Convergence}

Now we investigate relations between Luxemburg norm and the integral
\begin{equation*}
R_G(u):=\int_IG(u)\,dt.
\end{equation*}
A functional $R_G$ is called modular. Theory of modulars is well known and is developed in more
general setting than ours. More information can be found in \cite{Mus83,RaoRen02}.

For Lebesgue spaces a notions of modular and norm are indistinguishable because modular $\int_I
|u|^p\, dt$ is equal to $\Lpnorm[p]{u}^p$. But in Orlicz spaces relation between
$R_G$ and $\LGnorm{\cdot}$ is more complex.

There is remarkable difference between isotropic and anisotropic spaces. It is clear that if
$u,v\in \Lpspace$ (or more generally in isotropic Orlicz space) then $|u(t)|\leq |v(t)|$ a.e.
implies $\Lpnorm{u}\leq \Lpnorm{v}$. In anisotropic case it is no longer true, even if
$G(u(t)) < G(v(t))$. Next two examples illustrates this point.

\begin{example}
Let $G(x,y)=(x-y)^2+y^4$, $I=[0,1]$, $u(t)=(2,0)$ and $v(t)=(2,3/2)$. Then $|u(t)| < |v(t)|$,
$G(u(t)) < G(v(t))$ and $R_G(u)\leq R_G(v)$, but $2=\LGnorm{u} > \LGnorm{v} \simeq 1.6$.
\end{example}

\begin{example}
Let $G(x,y)=x^2+y^4$, $u(t)=(1,0)$ and $v(t)=\tfrac{11}{10}(\cos t,\sqrt{\sin t})$. In
$\LGspace([0,\pi],\R^2)$ we have
\[
\sqrt{\pi}=\LGnorm{u}>\LGnorm{v}\simeq 1.7
\]
but $|u(t)|< |v(t)|$,  $G(u(t)) < G(v(t))$ for all $t\in [0,\pi]$ and $R_G(u)<R_G(v)$.
\end{example}

\begin{definition}
    We say that a subset $K\subset\LGspace$ is modular bounded if there exists $C>0$ such that
    \[
    R_G(u)\leq C, \text{ for all } u\in K.
    \]

\end{definition}

Modular boundedness is sometimes called mean boundedness. It is evident that  $R_G(u) \leq
\LGnorm{u}$ if $\LGnorm{u}\leq 1$ and $R_G(u) > \LGnorm{u}$ if
$\LGnorm{u}> 1$.

\begin{lemma}
    \label{lem:modularboundness}
Let $u\in \LGspace$.
\begin{enumerate}
    \item If $R_G(u)\leq C$ then $\LGnorm{u}\leq \max\{C,1\}$.
    \item If $\LGnorm{u}\leq C$ then $R_G(u)\leq \mu(I) \widetilde{C} +K_1(C)$ for some
$\widetilde{C}>0$.
\end{enumerate}
Moreover, a set $K\subset\LGspace$ is modular bounded if and only if is norm bounded.
\end{lemma}
\begin{proof}
Assume that $R_G(u)\leq C$. If $C\leq 1$ then $\LGnorm{u}\leq 1$. If $C>1$ then
\[
    \int_I G\left(\frac{u}{C}\right)\,dt\leq \frac1C\int_IG(u)\,dt\leq 1.
\]
This implies $\LGnorm{u}\leq \max\{C,1\}$. For the second statement, assume $\LGnorm{u}\leq~C$.
Then
\[
    R_G(u)
    =
    \int_{I_1} G\left(u\right)\, dt
    +
    \int_{I\setminus I_1} G\left(C\,\frac{ u}{C}\right)\, dt
    \leq
    \mu(I_1)\, \widetilde{C}+K_1(C)\int_I G\left(\frac{u}{C}\right)\,dt,
\]
where $I_1=\{t\in I\colon |u(t)|\leq M_1 C\}$ and $\widetilde{C}>0$. To finish the proof
observe that
\[
    \int_I G\left(\frac{u}{C}\right)\,dt\leq \int_I G\left(\frac{u}{\LGnorm{u}}\right)\,dt=1.
\]
\end{proof}

\begin{definition}
We say that  a sequence of functions $u_k\in \LGspace$ is modular convergent to $u\in \LGspace$
if $R_G(u_k-u)\to 0$ as $k\to\infty$.
\end{definition}

Modular convergence is sometimes called mean convergence. Norm convergence always implies modular
convergence. Let $\LGnorm{u_k}\to 0$ as $k\to\infty$. We can assume that $\forall_{k}\
\LGnorm{u_k}\leq 1$, then
\begin{equation*}
\frac{1}{\LGnorm{u_k}}R_G(u_k)\leq R_G\Big(\frac{u_k}{\LGnorm{u_k}}\Big) = 1.
\end{equation*}
Hence $0\leq R_G(u_k)\leq \LGnorm{u_k}$. In general, converse is not true unless $G$ satisfies
$\Delta_2$ condition. (see \cite{KraRut61,Sch05}).

\begin{theorem}
\label{modularc=normc}
Norm convergence is equivalent to modular convergence.
\end{theorem}
\begin{proof}
We need only to prove that modular convergence implies norm convergence. Fix $\varepsilon>0$ and
assume that $\{u_k\}$ is modular convergent to $0$. Define
\[
I_{1,k}=\{t\in I\colon |u_k(t)|\leq M_1\}
\]
Since $G$ satisfies $\Delta_2$, for all $k>0$ we have
\begin{multline*}
\int_I G(u_k/\varepsilon)\, dt
\leq
\mu(I_{1,k})\,C_{M_1}+K_1(1/\varepsilon)\int_{I\setminus I_{1,k}} G(u_k)\, dt
\leq 
\mu(I)\,C_{M_1}+K_1(1/\varepsilon)\int_I G(u_k)\, dt.
\end{multline*}
For sufficiently large $k$ we have
\[
\int_I G(u_k)\, dt \leq \frac{1}{K_1(1/\varepsilon)}
\]
and
\[
\int_I G(u_k/\varepsilon)\, dt\leq \mu(I)\,C_{M_1}+1 = C.
\]
Finally, Lemma \ref{lem:modularboundness} shows that $\LGnorm{u_k}\leq C\varepsilon$ and hence
$\LGnorm{u_k}\to 0$.
\end{proof}

It is standard result due to Riesz that for $f_n$, $f\in\Lpspace$
\[
f_n\to f \text{ a.e. }\implies \Lpnorm{f_n}\to \Lpnorm{f} \iff \Lpnorm{f_n-f}\to 0.
\]
Following lemmas establish Orlicz space version of this fact.

\begin{lemma}
    \label{lem:BreLie}
    For every $k>1$ and $0<\varepsilon<\frac1k$ and $x,y\in\R^n$
\[
  |G(x+y)-G(x)|\leq \varepsilon|G(kx)-kG(x)|+2G(C_\varepsilon y)
\]
where $C_\varepsilon=\frac{1}{\varepsilon(k-1)}$
\end{lemma}
\begin{proof}
The proof is due to Brezis and Lieb \cite{BreLie83} (see also \cite{KhaKoz15}). We repeat the
proof. Let $\alpha=1-k\varepsilon$, $\beta=\varepsilon$, $\gamma=\varepsilon(k-1)$.
Then $\alpha+\beta+\gamma=1$ and $x+y=\alpha x+\beta (kx) + \gamma (C_\varepsilon y)$. By convexity
\[
  G(x+y)\leq \alpha G(x)+\beta G(kx)+\gamma G(C_\varepsilon y).
\]
This implies that
\[
  G(x+y)-G(x)\leq \varepsilon(G(kx)-kG(x))+ G(C_\varepsilon y).
\]
For the reverse inequality let
\[
  \alpha=\frac{1}{1+k\varepsilon},\quad \beta = \frac{\varepsilon}{1+k\varepsilon},\quad
  \gamma=\frac{\varepsilon(k-1)}{1+k\varepsilon}.
\]
Then $x=\alpha(x+y)+\beta (kx)+\gamma (-C_\varepsilon y)$ and
\[
  G(x)-G(x+y)\leq \varepsilon(G(kx)-kG(x))+ \varepsilon (k-1)G(C_\varepsilon y).
\]
\end{proof}

\begin{lemma}
    \label{lem:untou_implies_RGuntoRGu}
  If $u_n\to u$ in $\LGspace$ then $R_G(u_n)\to R_G(u)$.
\end{lemma}
\begin{proof}
    In Lemma \ref{lem:BreLie} set $x+y=u_n$, $x=u$, $k=2$. Then $\varepsilon<1/2$,
$C_\varepsilon=\frac1\varepsilon$ and
\[
  |G(u_n)-G(u)|\leq \varepsilon|G(2u)-2G(u)|+2G\left(\frac{u_n-u}{\varepsilon}\right).
\]
Since $u_n\to u$ in $\LGspace$, there exists $n_0$ such that for $n>n_0$ we have
$\LGnorm{u_n-u}<\varepsilon^2\leq \varepsilon < 1$. Thus
\[
\int_IG\left(\frac{u_n-u}{\varepsilon}\right)\,dt
\leq
\frac{1}{\varepsilon}\LGnorm{u_n-u}<\varepsilon.
\]
From this and inequality above we obtain
\[
  |R_G(u_n)-R_G(u)|\leq \varepsilon\int_I|G(2u)-2G(u)|\,dt+2\varepsilon.
\]
Letting $\varepsilon\to 0$ we have $R_G(u_n)\to R_G(u)$.
\end{proof}

Norm convergence $u_n\to u$ in $\Lpspace$ implies that there exists a subsequence such that
$u_{n_k}\to u$ a.e.  and $|u_{n_k}|\leq |h|\in \Lpspace$. According to the above lemma, if $u_n\to
u$
in $\LGspace$ then:
\begin{enumerate}
  \item Since $\LGspace\hookrightarrow \Lpspace[1]$ (see Lemma \ref{lem:LGembed} below), we can
  extract a subsequence $u_{n_k}$ such that
\[
  u_{n_k}\to u \text{ a.e and } |u_{n_k}|\leq h\in \Lpspace[1](I,\R).
\]
  \item Since $R_G(u_n-u)\to 0$, $G(u_n-u)\to 0$ in $\Lpspace[1]$. Thus we can extract a
  subsequence $\{u_{n_k}\}$ such that
\[
G(u_{n_k}-u)\to 0 \text{ a.e and } G(u_{n_k}-u)\leq h\in \Lpspace[1](I,\R).
\]
  \item Since $R_G(u_{n})\to R_G(u)$, $G(u_n)\to G(u)$ in $\Lpspace[1]$. Hence there exists a
  subsequence $\{u_{n_k}\}$ such that
\[
G(u_{n_k})\to G(u) \text{ a.e and } G(u_{n_k})\leq h\in \Lpspace[1](I,\R).
\]
\end{enumerate}

\begin{lemma}
    \label{lem:rieszforLG}
    Let $\{u_n\}\subset \LGspace$ and $u\in \LGspace$. Suppose that
    \begin{enumerate}
        \item $u_n\to u$ a.e.
        \item $R_G(u_n)\to R_G(u)$.
    \end{enumerate}
Then $u_n\to u$ in $\LGspace$.
\begin{proof}
This lemma was proved in \cite[p. 83]{RaoRen91} for N-functions. Since $G$ is convex, we get
$\frac12(G(u_n(t))+G(u(t)))-G(\frac{u_n(t)-u(t)}{2}) \geq 0$. Continuity of $G$ and $u_n\to u$ a.e.
implies
    \[
    \frac12(G(u_n(t))+G(u(t)))-G\left(\frac{u_n(t)-u(t)}{2}\right)\to G(u) \text{ a.e.}
    \]
So that by the Fatou Lemma, we have
\begin{multline*}
    \int_I G(u)\,dt
    \leq
    \liminf_{n\to \infty}\int_I \frac12(G(u_n)+G(u))\,dt-G\left(\frac{u_n-u}{2}\right)\, dt
    \leq \\ \leq
    \lim_{n\to \infty}\int_I\frac12(G(u_n)+G(u))\,dt-
    \limsup_{n\to \infty} \int_I G\left(\frac{u_n-u}{2}\right)\,dt
    =\\=
    \int_IG(u)\, dt-\limsup_{n\to \infty}\int_IG\left(\frac{u_n-u}{2}\right)\,dt.
\end{multline*}
This implies that
\[
    \int_I G\left(\frac{u_k(t)-u(t)}{2}\right)\,dt\to 0
\]
and $\LGnorm{u_k-u}\to 0$ by Theorem \ref{modularc=normc}.
\end{proof}
\end{lemma}

As a consequence we obtain dominated convergence theorem for anisotropic Orlicz spaces:

\begin{theorem}
    \label{thm:dominatedconvergence}
    Suppose that $\{u_n\}\subset \LGspace$ and
    \begin{enumerate}
      \item $u_n\to u$ a.e.
      \item there exists $h\in \Lpspace[1]$ such that $G(u_n)\leq h$ a.e.
    \end{enumerate}
Then $u\in \LGspace$ and $u_n\to u$ in $\LGspace$.
\end{theorem}
\begin{proof}
Since $G$ is continuous and $u_n\to u$ a.e., $G(u_n)\to G(u)$ a.e. It follows that $G(u)\leq h$
a.e. Thus $G(u)\in \Lpspace[1]$ and hence $u\in \LGspace$. Since $0\geq
h\pm G(u_n)$ and $h\pm G(u_n)\to h\pm G(u)$ a.e., application of the Fatou Theorem yields
\[
  \int_I h\,dt\pm \int_I G(u_n)\,dt\leq \liminf \int_I h\,dt\pm G(u_n)\,dt.
\]
Therefore,
\begin{align*}
  \int_I h\,dt + \int_I G(u)\,dt\leq \int_I h\,dt + \liminf \int_I G(u_n)\,dt  \\
  \int_I h\,dt - \int_I G(u)\,dt\leq \int_I h\,dt - \limsup \int_I G(u_n)\,dt
\end{align*}
and hence
\[
  \limsup \int_I G(u_n)\,dt \leq \int_I G(u)\,dt\leq \liminf \int_I G(u_n)\,dt
\]
and $R_G(u_n)\to R_G(u)$. By the Lemma \ref{lem:rieszforLG}, $u_n\to u$ in $\LGspace$.
\end{proof}

In the above theorem, assumption $G(u_n)\leq h$ can be replaced by $G(u_n)\leq G(h)$,
$h\in\LGspace$.
Consider a sequence $\{u_n\}\subset \LGspace$ convergent pointwise to measurable function $u$. Under
standard dominance condition (i.e.  $|u_n|\leq |g|$, $g\in \LGspace$) it is not true in general
that $u_n\to u \in \LGspace$.

\begin{example}
Let $G(x,y)=x^2+y^4$, $I=(0,1)$, $u(t)=(0,t^{-1/4}) \text{ and } h(t)=(t^{-3/8},0)$. Define
\[
u_n(t)=\begin{cases}
         u(t)&|u(t)|\leq n\\
         0 & |u(t)|> n
       \end{cases}
\]
Then $u_n\to u$ a.e., $u_n,h\in \LGspace$ and $|u_n|\leq |h|$ for every $t$. But
$G(u(t))=t^{-1}\notin \Lpspace[1](I,\R)$. Hence $u\notin \LGspace$.
\end{example}

\begin{remark}
    Modular $R_G$ is called monotone modular if $|x|\leq |y|$ implies $R_G(x)\leq R_G(y)$. If $R_G$
is monotone modular then $u_k\to u$ a.e and $|u_k|\leq |g|$, $g\in \LGspace$ implies $u\in
\LGspace$ and
$\LGnorm{u_k-u}\to 0$. We refer the reader to \cite{KhaKoz15} for more details.
\end{remark}

\subsection{Separability}

For every $u \in\LGspace$ there exists a sequence of bounded functions $\{u_n\}\subset \LGspace$
such that $u_n\to u$ in $\LGspace$. For example, one can define
\[
u_n(t)=\begin{cases}u(t)&|u(t)|\leq n\\0&|u(t)|>n\end{cases}
\]
In this case $u_n\to u$ a.e and $G(u_n(t)-u(t))\leq G(u(t))$. Therefore, by Theorem
\ref{thm:dominatedconvergence} we get $u_n\to u$ in $\LGspace$.

\begin{theorem}[{cf. \cite[p. 81]{KraRut61}}]
  The space $\LGspace$ is separable.
\end{theorem}
\begin{proof}
Fix $\varepsilon>0$. Suppose that $u\in \LGspace$ is bounded and $|u(t)|\leq a$. Set
$C=\sup\{G(x/\varepsilon)\colon |x|\leq 2a\}$.

By the Luzin theorem we can find a compact subset $I_1\subset I$ and a continuous function
$u_1\colon I\to \R^N$ such that $\mu(I\setminus I_1)\leq 1/C$, $u(t)=u_1(t)$ for all $t\in I_1$ and
$|u_1(t)|\leq a$. Now we get
\[
\int_I G\left(\frac{u-u_1}{\varepsilon}\right)\,dt
=
\int_{I\setminus I_1} G\left(\frac{u-u_1}{\varepsilon}\right)\,dt
\leq \mu(I\setminus I_1) C\leq 1.
\]
Hence $\LGnorm{u-u_1}\leq \varepsilon$. For arbitrary $v\in \LGspace$ we can find a bounded $u_1\in
\LGspace$ such that $\LGnorm{v-u}\leq \varepsilon/2$. Thus
\[
  \LGnorm{u-u_1}\leq \varepsilon.
\]
For every continuous function there exists uniformly convergent sequence of polynomials with
rational coefficients.  It is easy to check that uniform convergence implies norm convergence in
$\LGspace$. This completes the proof.
\end{proof}

\begin{remark}
It is well known that if G-function does not satisfies $\Delta_2$ condition then $\LGspace$ is
not separable. One can define a subspace $E^G$ as  the closure of bounded functions
under Luxemburg norm. In this case, the space $E^G$ is a proper subset of $\LGspace$ and is always
separable (see \cite{KraRut61,Sch05}).
\end{remark}

\subsection{Embeddings}
We will use the symbols $\hookrightarrow$ nad $\hookrightarrow\hookrightarrow$ for, respectively,
continuous and compact embeddings. Recall that
\[
  F\prec G \iff F(x)\leq G(Kx), |x|\geq M.
\]
and
\[
F\prec\prec G \iff \lim_{x\to\infty} \frac{G(\alpha x)}{G(x)}=\infty, \text{ for all } \alpha>0.
\]

Next two theorems provide a basic embeddings for Orlicz spaces.

\begin{proposition}
Assume that $F\prec G$. Then $L^G\hookrightarrow L^F$
and
\[
\LGnorm[F]{u}\leq K(C \mu(I)+1)\LGnorm{u}.
\]
for some $C>0$.
\end{proposition}
\begin{proof}
It is evident that $\LGspace\subset \LGspace[F]$. Let $u\in \LGspace$ and set
\[
I_1=\Big\{t\in I: \Big|\frac{u(t)}{K\LGnorm{u}}\Big|\leq M\Big\}
\]
For every $t\in I_1$, we have
\[
F\Big(\frac{u(t)}{K\LGnorm{u}}\Big)\leq G\Big(\frac{u(t)}{\LGnorm{u}}\Big)
\]
and
\begin{multline*}
\int_I F\Big(\frac{u}{K\LGnorm{u}}\Big)dt
=
\int_{I\setminus I_1}F\Big(\frac{u}{K\LGnorm{u}}\Big)dt
+
\int_{I_1}F\Big(\frac{u}{K\LGnorm{u}} \Big)dt
\leq \\ \leq
    \mu(I\setminus I_1)\widetilde{C}+
    \int_I G\Big(\frac{u}{\LGnorm{u}}\Big)dt
\leq
    \mu(I)\widetilde{C}+1,
\end{multline*}
where $\widetilde{C}=\sup\{G(x)\colon |x|\leq M\}$. Since  $1\leq \widetilde{C}\mu(I)+1$, we have
\[
\int_I F\Big(\frac{u(t)}{K(\widetilde{C}\mu(I)+1)\LGnorm{u}}\Big)dt\leq 1.
\]
Finally,
\[
\LGnorm[F]{u}\leq K(\widetilde{C}\mu(I)+1)\LGnorm{u}.
\]

\end{proof}

It is easy to see that there exist constants $C_1,C_2>0$ such that $\Lpnorm[1]{u}\leq C_1
\LGnorm{u}$ and  $\LGnorm{u}\leq C_2 \Lpnorm[\infty]{u}$.

Directly from Lemma \ref{lemma:xp<G<xq} we obtain that Orlicz spaces can be viewed as
a spaces between two Lebesgue spaces determined by constants in $\Delta_2$ and $\nabla_2$
conditions.

\begin{proposition}
    \label{lem:LGembed}
    For every $G$ there exists $p,q\in (1,\infty)$ such that
    \[
    \Lpspace[q]\hookrightarrow \LGspace \hookrightarrow \Lpspace[p].
    \]
In particular $\Lpspace[\infty]\hookrightarrow \LGspace \hookrightarrow\hookrightarrow \Lpspace[1]$.
\end{proposition}

\begin{theorem}[{cf. \cite[th. 8.25]{Ada75}}]
  If $F\prec\prec G$ then $\LGspace\hookrightarrow\hookrightarrow \LGspace[F]$.
\end{theorem}
\begin{proof}
    Let $\{u_n\}$ be a bounded sequence in $\LGspace$. Since $\LGspace\hookrightarrow \LGspace[F]
    \hookrightarrow\hookrightarrow \Lpspace[1]$, $\{u_n\}$ is bounded in $\LGspace[F]$ and  there
    exists a subsequence, denoted again by $\{u_n\}$, convergent in $\Lpspace[1]$. Hence
    $\{u_n\}$ converges in measure and thus is Cauchy in measure.

Fix $\varepsilon>0$ and let $v_{n,k}(t)=(u_n(t)-u_k(t))/\varepsilon$. Since $\{u_n\}$ is bounded in
$\LGspace[F]$, there exist $C>0$ such that $\LGnorm[F]{u_n}\leq C$. There exist $M>0$ such that if
$|x|\geq M$, then
\[
  F(x)\leq \frac12 G\left(\frac{x}{C}\right).
\]
Set $\overline{F}(M)=\sup\{F(x)\colon |x|\leq M\}$,
\[
I_{n,k}=\left\{t\in I\colon F(v_{n,k})\geq \frac{1}{\mu(I)}\right\}
,\
I'_{n,k}=\{t\in I\colon |v_{n,k}(t)|\geq M\}
,\
I''_{n,k}=I_{n,k}\setminus I'_{n,k}.
\]
Since $\{u_n\}$ is Cauchy in measure, there exists $N$ such that if $n$, $k\geq N$, then
$\mu(I''_{n,k})\leq \mu(I_{n,k})\leq \frac{1}{2\overline{F}(M)}$. Observe that
\begin{enumerate}
  \item if $t\in I\setminus I_{n,k}$ then $F(v_{n,k}(t))\leq 1/2\mu(I)$,
  \item if $t\in I'_{n,k}$, then $F(v_{n,k}(t))\leq \tfrac14 G(v_{n,k}/C)$,
  \item if $t\in I''_{n,k}$, then $F(v_{n,k}(t))\leq \overline{F}(M)$.
\end{enumerate}
It follows that for $n$, $k\geq N$, we have
\begin{multline*}
\int_I F(v_{n,k})\,dt
=
\left(\int_{I\setminus I_{n,k}}+\int_{I'_{n,k}}+\int_{I''_{n,k}}\right)F(v_{n,k})\,dt
\leq \\ \leq
    \frac{\mu(I)}{2\mu(I)}
    +
    \frac{1}{4}\int_I G\left(\frac{v_{n,k}}{C}\right)\,dt
    +
    \frac{1}{2\overline{F}(M)} \overline{F}(M)
\leq 1.
\end{multline*}

Hence $\LGnorm[F]{u_n-u_k}\leq \varepsilon$ and so $\{u_n\}$ converges in $\LGspace[F]$.
\end{proof}

 In some cases, $\LGspace$ is simply a product of $\Lpspace[p_i](I,\R)$, but there exists
Orlicz spaces which are not in the form $\Lpspace[p](I,\R)\times \Lpspace[q](I,\R)$ (cf. \cite[pp.
18-20]{Tru74}).

\begin{example}
Consider the Orlicz space $\LGspace=\LGspace(I,\R^2)$ generated, by $G(x)=|x_1|^{p_1}+|x_2|^{p_2}$,
$p_1,p_2>0$. If $u=(u_1,u_2)\in \Lpspace[p_1](I,\R)\times \Lpspace[p_2](I,\R)$, then
\[
\int_I G(u)\,dt=\int_I |u_1|^{p_1}\,dt+\int_I |u_2|^{p_2}\,dt <\infty.
\]
Conversely, if $u=(u_1,u_2)\in \LGspace$ then
\[
\int_I |u_1|^{p_1}\,dt \leq \int_I G(u)\,dt <\infty
  \text{ and }
  \int_I |u_2|^{p_2}\,dt  \leq \int_I G(u)\,dt <\infty.
\]
Hence $u\in \Lpspace[p_1](I,\R)\times \Lpspace[p_2](I,\R)$.
\end{example}

\begin{example}
    \label{ex:LGnotLPtimesLP}
Consider the Orlicz space $\LGspace=\LGspace(I,\R^2)$ generated, by  $G(x)=(x_1-x_2)^4+x_2^2$.
From Lemmas \ref{lemma:xp<G<xq} and \ref{lem:LGembed} we obtain that
$\Lpspace[4](I,\R^2)\hookrightarrow\LGspace\hookrightarrow\Lpspace[2](I,\R^2)$. Let $u_1$
be a function in $\Lpspace[2](I,\R)$ such that $u_1\notin \Lpspace(I,\R)$, for $p>2$. Set
$u=(u_1,u_1)$, then
\[
  \int_I G(u)\,dt=\int_I |u_1|^2\,dt <\infty
\]
but
\[
  \int_I |u|^p\,dt =\infty.
\]
Therefore for every $p>2$ there exists $u\in \LGspace$ such that $u\notin \Lpspace(I,\R^2)$.
Moreover, $u\notin \Lpspace(I,\R)\times \Lpspace[2](I,\R)$ for any $p>2$. From the other
hand if $u=(u_1,u_2)\in \Lpspace[4](I,\R)\times \Lpspace[4](I,\R)$ then $u\in \LGspace$.
Therefore
\[
  \Lpspace[4](I,\R)\times \Lpspace[4](I,\R) \hookrightarrow
  \LGspace
  \hookrightarrow
  \Lpspace[2](I,\R)\times \Lpspace[2](I,\R)
\]
but $\LGspace$ cannot be identified with any
\[
\Lpspace[4](I,\R)\times \Lpspace[4](I,\R)
\hookrightarrow\Lpspace[p](I,\R)\times \Lpspace[q](I,\R)\hookrightarrow
  \Lpspace[2](I,\R)\times \Lpspace[2](I,\R).
  \]
\end{example}

\subsection{Duality}

Since $\LGspace\hookrightarrow \Lpspace \hookrightarrow\hookrightarrow \Lpspace[{p_0}]
\hookrightarrow \Lpspace[1]$
($p$ given by $\nabla_2$) and $1<p_0<p$, it follows that $\LGspace$ is closed subspace of reflexive
space. Therefore $\LGspace$ is reflexive itself.

\begin{theorem}
  $\LGspace$ is a reflexive Banach space.
\end{theorem}

The rest of this section is devoted to proving  that the general formula
for bounded linear operator $F\colon \LGspace\to \R$ is
\[
  F(u)=\int_I\inner{u}{v}\,dt,
\]
where $v\in \LGspace[G^\ast]$. We show that the dual space  $(\LGspace)^\ast$ can be identified
with the Orlicz space $\LGspace[{G^\ast}]$ generated by conjugate function $G^\ast$. On the other
hand,  $(G^\ast)^\ast=G$ and $(\LGspace)^\ast\simeq \LGspace[G^\ast]$ implies reflexivity as well.

\begin{lemma}
    Every $v\in \LGspace[{G^\ast}]$ can be identified with the following functional $F_v\in
    (\LGspace)^\ast$:
    \[
    F_v(u)=\int_I\inner{u}{v}\, dt.
    \]
    Moreover $\|F_v\|\leq 2\LGastnorm{v}$.
\end{lemma}
\begin{proof}
    It is easy to see that $F_v$ is linear. By the H\"older inequality we get
    \[
    F_v(u) = \int_I\inner{u}{v}\, dt\leq 2\LGnorm{u}\LGastnorm{v}.
    \]
    Thus $F_v$ is bounded and $\|F_v\|\leq 2\LGastnorm{v}$.
\end{proof}

\begin{lemma}[cf. \cite{DesGri01,Sch05}]
    \label{lemma:vinLGast}
    If $v\in \Lpspace[1](I,\R^n)$ is such that for each piecewise constant function $u\in \LGspace$
satisfy
    \[
        \int_I\inner{u}{v}\, dt\leq M\LGnorm{u},
    \]
    then $v\in \LGspace[{G^\ast}]$ and $\LGnorm[{G^\ast}]{v}\leq M$.
\end{lemma}
\begin{proof}
Define an approximation
 \[
 v_{n,i}=\frac{n}{\mu(I)}\int_{E_i}v\,dt,\quad E_i \text{ - disjoint},\quad I=\bigcup_{i=1}^n
E_i,\quad
\mu(E_i)=\frac{\mu(I)}{n}.
 \]
Set $v_n=\sum_{i=1}^n v_{n,i}\chi_{E_i}$. Let $u\in \LGspace$ be a simple function, define
approximation $u_{n}$ of $u$ in the same way. By Jensen inequality
\begin{multline*}
    \int_I G\left(\frac{u_n}{\LGnorm{u}}\right)\,dt
    =
    \sum_{i=1}^n \mu(E_i)G\left( \frac{1}{\mu(E_i)}\int_{E_i}\frac{u}{\LGnorm{u}}\,dt \right)
    \leq \\ \leq
    \sum_{i=1}^n\mu(E_i)\frac{1}{\mu(E_i)}\int_{E_i} G(\frac{u}{\LGnorm{u}})\,dt
    =
    \int_IG(\frac{u}{\LGnorm{u}})\,dt
    =1.
\end{multline*}
Hence $\LGnorm{u_n}\leq \LGnorm{u}$. A direct computation yields

\[
    \int_I\inner{u}{v_n}\,dt=\int_I\inner{u_n}{v}\,dt \leq M\LGnorm{u_n}\leq M\LGnorm{u}.
\]
We can find for each $v_{n,i}$ a $z_{n,i}\in \R^n$ such that
$
\inner{ z_{n,i}}{v_{n,i}/M}=G(z_{n,i})+G^\ast(v_{n,i}/M).
$
Suppose that $\sum_{i=1}^n \mu(E_i) G(z_{n,i})>1$. Then there exists $\beta<1$ such that
$\sum_{i=1}^n \mu(E_i) G(\beta z_{n,i})=1$. Putting

\[
u=\sum_{i=1}^n \beta z_{n,i}\chi_{E_i}
\]
we obtain that $\int_IG(u)\,dt\leq 1$ and $\LGnorm{u}\leq 1$. Therefore
\begin{multline*}
\sum_{i=1}^n \mu(E_i) G^\ast(v_{n,i}/M)
=
\frac{1}{M}\sum_{i=1}^n \mu(E_i) \inner{z_{n,i}}{v_{n,i}}-\sum_{i=1}^n \mu(E_i) G(z_{n,i})
= \\ =
\frac{1}{M\beta}\int_I\inner{u}{v_n}\,dt-\sum_{i=1}^n \mu(E_i)G(z_{n,i})
\leq
\frac{1}{\beta}-\frac{1}{\beta}\sum_{i=1}^n \mu(E_i) G(\beta z_{n,i})\leq 0.
\end{multline*}
Now assume that $\mu(E_i)\sum G(z_{n,i})\leq 1$ and repeat the same computation with $\beta=1$ and
obtain
\[
\sum_{i=1}^n \mu(E_i) G^\ast(v_{n,i}/M)\leq 1.
\]
In both cases we get
\[
    \int_I G^\ast(v_n/M)\,dt\leq 1.
\]
Since $v_n\to v$ a.e.  we can conclude that $G^\ast(v/M)\leq\lim G^\ast(v_n/M)$. By the Fatou
theorem we get
\[
    \int_IG^\ast(v/M)\,dt\leq 1.
\]
\end{proof}

\begin{lemma}[cf. \cite{KraRut61,Sch05}]
For every $F\in (\LGspace)^\ast$ there exists unique $v\in \LGspace[{G^\ast}]$ such that for
every $u\in \LGspace$
\[
    Fu=\int_I\inner{u}{v}\, dt.
\]
\end{lemma}
\begin{proof}
For a measurable subset $E\subset I$ define $\chi^N_E(x)=(\chi_E,\dots,\chi_E)$. Note that
$\chi^N_E\in \LGspace$. Set
\[
    \phi(E)=F\left(\chi^N_E\right).
\]
For every sequence $\{E_i\}$ of measurable and pairwise disjoint subsets of $I$ such that
$E=\bigcup E_i$ we have $\chi^N_E=\sum \chi^N_{E_i}$ and
\[
\phi\left(E\right)=F\left(\chi^N_E\right)=F\left(\sum\chi^N_{E_i}\right)=\sum
F\left(\chi^N_{E_i}\right)=\sum \phi\left({E_i}\right).
\]

Suppose that there exists a sequence $\{E_i\}$ of measurable sets and $\delta>0$ such that
$\mu(E_i)\to 0$ and $\LGnorm{\chi^N_{E_i}}>\delta$ for all $i$. Then
\[
    1<\int_IG\left(\frac{\chi^N_{E_i}}{\delta}\right)\,dt
    =
    \int_{E_i} G\left(\frac{(1,\dots,1)}{\delta}\right)\,dt=
    \mu(E_i)G\left(\frac{(1,\dots,1)}{\delta}\right).
\]
A contradiction. From inequality
\[
    |\phi(E_i)|\leq \|F\| \LGnorm{\chi^N_{E_i}}
\]
we obtain that if $\mu(E_i)\to 0$ then $|\phi(E_i)|\to 0$. Thus a set function $\phi$ is
$\sigma$-additive and absolutely continuous with respect to Lebesgue measure.

It follows from the Radon-Nikodym theorem that there exists a function $v\in \Lpspace[1](I,\R^N)$
such that
\[
    F(\chi^N_E)=\phi(E)=\int_I\inner{\chi^N_E}{v}\,dt.
\]
For every step function $u=\sum c_i\chi_{E_i}$, by linearity of $F$,
\[
F(u)=F(\sum c_i\chi_{E_i})=\sum c_iF(\chi_E)=\sum
c_i\int_I\inner{\chi_E}{v}\,dt=\int_I\inner{u}{v}\,dt.
\]

By lemma \ref{lemma:vinLGast} we get that $v\in \LGspace[{G^\ast}]$. Assume now that $u$ is
bounded. Choose a sequence of step functions $\{u_n\}$ such that
\[
u_n(t)=\sum c_i \chi_{E_i},\quad  c_i = \frac{1}{\mu(E_i)}\int_{E_i} u\, dt,
\]
where $E_i$ are disjoint and
\[
\mu(E_i)=\frac{\mu(I)}{n},\quad I=\bigcup_{i=1}^n E_i.
\]
Clearly, $u_n\to u$ a.e. and the sequence $\{u_n\}$ is uniformly bounded. It follows that
\[
F(u)=\lim_{n\to\infty} F(u_n)=\lim_{n\to \infty} \int_I\inner{u_n}{v}\,dt=\int_I\inner{u}{v}\,dt.
\]
Suppose that $u$ is an arbitrary function in $\LGspace$. There exists a sequence $\{u_n\}$ of
bounded functions which converges a.e. to $u$ such that $|u_n(t)|\leq |u(t)|$ a.e.  Thus
\[
F(u)=\lim F(u_n)=\lim_{n\to\infty} \int_I\inner{u_n}{v}\,dt = \int_I\inner{u}{v}\,dt.
\]
It remains to show that
$v$ is unique. Suppose that $v_1$ and $v_2$ represent $F$. Then we have
\[
    \int_I\inner{u}{v_1}\,dt=\int_I\inner{u}{v_2}\,dt
\]
for all $u\in \Lpspace[\infty]$. Thus $v_1=v_2$.
\end{proof}

As a consequence we obtain that $\LGspace[G^\ast]\simeq (\LGspace)^\ast$.
Since $G^{\ast\ast}=G$, we also get $\LGspace \simeq (\LGspace[G^\ast])^\ast$.

\begin{remark}
If G-function does not satisfies $\Delta_2$ condition then
$\LGspace$ is not reflexive and $(\LGspace)^\ast$ is not isomorphic to $\LGspace[G^\ast]$
(see
\cite{KraRut61,Sch05}).
\end{remark}

\section{Orlicz-Sobolev spaces}
\label{sec:OSspace}

The Orlicz-Sobolev space $ \WLGspace= \WLGspace(I ,\R^n)$ is defined to be
\begin{equation*}
 \WLGspace(I ,\R^n):=\{u\in \LGspace(I ,\R^n) : \dot{u}\in \LGspace(I ,\R^n)\}.
\end{equation*}
For $u\in \WLGspace$ we define
\begin{equation*}
 \WLGnorm{u}:= \LGnorm{u}+\LGnorm{\dot{u}}.
\end{equation*}

Define  $\WzLGspace= \WzLGspace(I,\R^n)$ as the closure of $C^1_0(I,\R^n)$ in $\WLGspace$ with
respect to the $\WLGnorm{\cdot}$.

\begin{theorem}
    \label{thm:W1LG_reflexive}
    The space $(\WLGspace,\WLGnorm{\cdot})$ is a separable reflexive Banach space.
\end{theorem}

Proof is standard and will be omitted, see for instance \cite{Bre11}. If $G(x)=\frac1p|x|^p$, then
the Orlicz-Sobolev space $\WLGspace$ coincides with the Sobolev space $\Wspace{1,p}(I,\R^n)$.
Observe that $u_n\to u$ in $\WLGspace$ is equivalent to $R_G(u_n-u)\to 0$ and
$R_G(\dot{u}_n-\dot{u})\to 0$.

On $\WLGspace$ one  can introduce another norm (cf. \cite{MihRad07}):
\[
\WLGnormother{u}=\inf\{\alpha>0\colon
\int_I G\left(\frac{u}{\alpha}\right)
+
G\left(\frac{\dot{u}}{\alpha}\right)\, dt\leq 1
\}.
\]

\begin{proposition}
 A function $\WLGnormother{\cdot}$ is an equivalent  norm on $\WLGspace$. Moreover
 \[
 \WLGnorm{u}\leq 2 \WLGnormother{u}\leq 4\WLGnorm{u}.
 \]
\end{proposition}
\begin{proof}
  The proof that $\WLGnormother{\cdot}$ is a norm is similar to the proof of Theorem
\ref{th:Luxemburgisnorm} and is left to the reader. For the other part, note that
\[
\int_I G\left(\frac{u}{\WLGnormother{u}}\right)
+
G\left(\frac{\dot{u}}{\WLGnormother{u}}\right)\, dt  \leq 1
\]
implies
\[
\int_I G\left(\frac{u}{\WLGnormother{u}}\right)\, dt \leq 1
\text{ and }
\int_I G\left(\frac{\dot{u}}{\WLGnormother{u}}\right)\, dt  \leq 1.
\]
From this $\LGnorm{u}\leq \WLGnormother{u}$ and $\LGnorm{\dot{u}}\leq \WLGnormother{\dot{u}}$ and
finally, $\WLGnorm{u}\leq 2 \WLGnormother{u}$.
Let $\alpha = \max\{\LGnorm{u},\LGnorm{\dot{u}}\}$. Since $\LGnorm{u},\LGnorm{\dot{u}}\leq \alpha$,
\[
  G\left(\frac{u(t)}{\alpha}\right)\leq \frac{\LGnorm{u}}{\alpha}
G\left(\frac{u(t)}{\LGnorm{u}}\right)
\]
and
\[
G\left(\frac{u(t)}{\alpha}\right)\leq \frac{\LGnorm{\dot{u}}}{\alpha}
G\left(\frac{u(t)}{\LGnorm{\dot{u}}}\right).
\]
Using the above relations, we obtain
\begin{multline*}
  \int_I G\left(\frac{u}{2\alpha}\right)
+
G\left(\frac{\dot{u}}{2\alpha}\right)\, dt
\leq
\frac12 \int_I G\left(\frac{u}{\alpha}\right)
+
G\left(\frac{\dot{u}}{\alpha}\right)\, dt
\leq \\ \leq
\frac12 \frac{\LGnorm{u}}{\alpha} \int_I G\left(\frac{u}{\alpha}\right)\,dt
+
\frac12 \frac{\LGnorm{\dot{u}}}{\alpha} G\left(\frac{\dot{u}}{\alpha}\right)\, dt
\leq 1
\end{multline*}
This implies $\WLGnormother{u}\leq 2 \alpha \leq 2\WLGnorm{u}$
\end{proof}

Since there exist $p,q\in(1,\infty)$ such that $\Lpspace[q]\hookrightarrow \LGspace\hookrightarrow
\Lpspace$, the following continuous embeddings exist
\begin{equation*}
  \Wspace{1,q}\hookrightarrow \WLGspace\hookrightarrow \Wspace{1,p}
\end{equation*}
Using standard results from the theory of Sobolev spaces we get
\begin{enumerate}
        \item $\WLGspace(I,\R^n)\hookrightarrow\hookrightarrow \Wspace{1,1}$
        \item $\WLGspace(I,\R^n)\hookrightarrow\hookrightarrow \Lpspace[q]$, for all $1\leq q \leq
\infty$
        \item $\WLGspace(I,\R^n)\hookrightarrow\hookrightarrow C(\overline{I})$
    \end{enumerate}

As a consequence we have

\begin{theorem}
    A function $u\in W^1L^G$ is absolutely continuous. Precisely, there exist absolutely
continuous representative of $u$ such that for all $a,b\in I$
\begin{equation*}
    u(b)-u(a)=\int_a^b\dot{u}(t)dt.
\end{equation*}
\end{theorem}

Directly from definition of $\WzLGspace$ we obtain important property of functions in $\WzLGspace$.
\begin{theorem}
If $u\in \WzLGspace$ then $u=0$ on $\partial I$.
\end{theorem}

Using embeddings mentioned above we have for every $u\in \WLGspace$
\begin{equation}
    \label{eq:Linf<WLG}
\Lpnorm[\infty]{u}\leq C \WLGnorm{u}.
\end{equation}

\begin{theorem}[Sobolev inequality]
For every function $u\in \WLGspace$
\[
    \LGnorm{u-u_I}\leq \mu(I)\LGnorm{\dot{u}}
\]
where $u_I=\frac{1}{\mu(I)}\int_Iu$.
\end{theorem}
\begin{proof}
Since $u$ is absolutely continuous, there exists $t_0\in I$ such that
$u(t_0)=\frac{1}{\mu(I)}\int_I u$ and for every $t\in I$ we have
    \[
    u(t)-u(t_0)=\int_{t_0}^t \dot{u}\, dt.
    \]
    By Jensen's inequality,
    \begin{multline*}
        G\left(\frac{u(t)-u(t_0)}{\mu(I)\LGnorm{\dot{u}}}\right)
        =
        G\left(\frac{1}{|t-t_0|}\int_{t_0}^t
        \frac{|t-t_0|}{\mu(I)}\frac{\dot{u}}{\LGnorm{\dot{u}}}\,dt\right)
        \leq\\\leq
        \frac{1}{|t-t_0|}
        \int_{t_0}^t
            G\left(
                \frac{|t-t_0|}{\mu(I)} \frac{\dot{u}}{\LGnorm{\dot{u}}}
            \right)\,dt
        \leq
        \frac{1}{\mu(I)} \int_I
            G \left(
                \frac{\dot{u}}{\LGnorm{\dot{u}}}
              \right)\,dt
        \leq \frac{1}{\mu(I)}.
    \end{multline*}
    Integrating both sides over $I$ we get
    \[
    \int_I G\left(\frac{u-u(t_0)}{\mu(I)\LGnorm{\dot{u}}}\right)\,dt\leq 1.
    \]
    Thus $\LGnorm{u-u_I}\leq \mu(I)\LGnorm{\dot{u}}$
\end{proof}

In similar way we get

\begin{theorem}[Poincare inequality]
 For every $u\in \WzLGspace$
 \[
  \LGnorm{u}\leq \mu(I) \LGnorm{\dot{u}}.
 \]
\end{theorem}
It follows that one can introduce equivalent norm in $\WzLGspace$:
\[
\WzLGnorm{u}=\LGnorm{\dot{u}}.
\]
Every linear functional $F$ on $\WzLGspace$ can be represented in the form
  \[
  F(u)=\int_I \inner{u}{v_0}+\inner{\dot{u}}{v_1}\,dt.
  \]
Where $v_0,v_1\in \LGspace[G^\ast]$. Moreover,
\[
\norm{F}=\max\{\LGnorm[G^\ast]{v_0},\LGnorm[G^\ast]{v_1}\}.
\]
In the case of Sobolev space $\Wspace{1,p}$ the proof is given in \cite[proposition 8.14]{Bre11},
but it remains the same for Orlicz-Sobolev spaces. As was pointed out in \cite{Bre11}, the first
assertion of the above proposition holds for every linear functional on $\WLGspace$.

\section{Variational setting}
\label{sec:functional}

In this section we examine the principal part
\begin{equation}
    \label{eq:functional}
\ISymbol(u)=\int_I F(t,u,\dot{u})\, dt
\end{equation}
of the variational functional associated with Euler-Lagrange equation
\[
\frac{d}{dt} F_v(t,u,\dot{u})=F_x(t,u,\dot{u})+\nabla V(t,u), \quad t\in I
\]
where $u\colon I\to \R^N$ and the Lagrangian $L\colon I\times \R^N\times\R^N\to \R$ is given by
$L(t,x,v)=F(t,x,v)+V(t,x)$.

In definition of the Orlicz space we need not to assume that $G$ is differentiable, but when we
consider the functional $\ISymbol$ we need it to show that $\ISymbol\in C^1$. Throughout this
section we will assume, in addition to \ref{G:0in0}--\ref{G:nabla2}, that $G$ satisfies

\begin{enumerate}[resume*=G]
    \item \label{G:class}
    $G$ is of a class $C^1.$
\end{enumerate}

\begin{remark}
Differentiability of $f$ is not sufficient to differentiability of $f^\ast$. But if $f$ is
finite, strictly convex, 1-coercive and differentiable then so is $f^\ast$. This result is in
close relation with Legendre duality (see \cite[p. 239]{HirLem04} and \cite{MawWil89} for more
details).
\end{remark}

It is well known that if $G$ is continuously differentiable then for all $x,y\in \R^n$
\begin{equation}
\label{eq:G-G<inner}
           G(x)-G(x-y)\leq \inner{\nabla G(x)}{y}\leq G(x+y)-G(x)
\end{equation}
and
\begin{equation*}
    \label{eq:G+Gast=inner}
    \inner{x}{\nabla G(x)}=G(x)+G^\ast(\nabla G(x)).
\end{equation*}
Let $y=x$ in \eqref{eq:G-G<inner}. Then
$\inner{\nabla G(x)}{x}\leq G(2x)-G(x)$.
Therefore, for all $x\in \R^N$
\[
  G^{\ast}(\nabla G(x))\leq G(2x).
\]
Directly from the above we get
\begin{proposition}
\label{eq:nablaGinG*}
If $u\in \LGspace$ then $\nabla G(u)\in \LGspace[G^\ast]$.
\end{proposition}

\begin{lemma}[{cf. \cite[lemma A.5]{Cletal04}}]
    \label{lem:convRG*nabla}
If $u_n\to u$ in $\LGspace$ then $R_{G^\ast}(\nabla G(u_n))\to R_{G^\ast}(\nabla G(u)).$

\end{lemma}
\begin{proof}
There exists a subsequence $\{u_{n_k}\}$ such that $u_{n_k}\to u$ a.e., $G(u_{n_k})\to G(u)$
a.e. and $G(u_{n_k})\leq h\in \Lpspace[1](I,\R)$. By continuity of $\nabla G$ and $G^\ast$ we have
$\nabla
G(u_{n_k})\to \nabla G(u)$ a.e. and
\[
G^\ast(\nabla G(u_{n_k}))\to G^\ast(\nabla G(u)) \text{ a.e.}
\]
Since $G^\ast(\nabla G(x))\leq G(2x)$,
\[
  G^\ast(\nabla G(u_{n_k}))\leq G(2u_{n_k})\leq C+K_1 G(u_{n_k}) \leq C+K_1h.
\]
By dominated convergence theorem $R_{G^\ast}(\nabla G(u_{n_k}))\to R_{G^\ast}(\nabla G(u))$. Since
this holds for any subsequence of $\{u_n\}$ we have that
\[
R_{G^\ast}(\nabla G(u_n))\to R_{G^\ast}(\nabla G(u)).
\]
\end{proof}

As a direct consequence of the above lemma and Lemma \ref{lem:rieszforLG} we obtain

\begin{proposition}
\label{eq:convG->convnablaG*}
\begin{equation*}
\LGnorm{u_n-u}\to 0 \implies \LGnorm[G^\ast]{\nabla G(u_n)- \nabla G(u)}\to 0.
\end{equation*}
\end{proposition}

\subsection{Case I}
We shall first examine a special case $F(t,x,v)=G(v)$, now functional \eqref{eq:functional} takes
the form
\[
\ISymbol(u)=\int_IG(\dot{u})\,dt.
\]
\begin{theorem}
    \label{thm:IC1}
    $\ISymbol\in C^1(\WLGspace,\R)$. Moreover
    \begin{equation}
    \label{I'}
    \ISymbol'(u)v=\int_I\inner{\nabla G(\dot{u})}{\dot{\varphi}}dt.
    \end{equation}
\end{theorem}
\begin{proof}
The proof follows similar lines as \cite[th. 3.2]{AciBurGiuMazSch15} (see also
\cite[thm 1.4]{MawWil89}).
First, note that $\dot{u}\in\LGspace$ implies
\[
0\leq\ISymbol(u)<\infty.
\]
It suffices to show that $\ISymbol$ has at every point $u$ directional derivative
$\ISymbol'(u)\in(\WLGspace)^{\ast}$ given by \eqref{I'} and that the mapping
$\ISymbol':\WLGspace\to(\WLGspace)^{\ast}$,  is continuous.

Let $u\in\WLGspace$, $\varphi\in\WLGspace\setminus\{0\}$, $t\in I$, $s\in[-1,1]$ . Define
\[
  H(s,t):=G(\dot{u}(t)+s\dot{\varphi}(t)).
\]
By \eqref{eq:G-G<inner} we obtain
\[
\int_I|H_s(s,t)|\,dt=\int_I |\inner{\nabla G(\dot{u}+s\dot{\varphi})}{\dot{\varphi}}|\,dt
\leq
\int_I G(\dot{u}+(s+1)\dot{\varphi}) + \int_I G(\dot{u}+s\dot{\varphi})\,dt <\infty.
\]
Consequently, $\ISymbol$ has a directional derivative and
\[
\ISymbol'(u)\varphi=\frac{d}{ds}\ISymbol(u+s\varphi)\Big|_{s=0}=\int_I\inner{\nabla
G(\dot{u})}{\dot{\varphi}}dt.
\]
By Lemma \ref{eq:nablaGinG*} and  H\"older inequality
\[
|\ISymbol'(u)\varphi|
=
\Big|\int_I\inner{\nabla G(\dot{u})}{\dot{\varphi}}dt\Big|
\leq
2\LGnorm[G^{\ast}]{\nabla G(\dot{u})}\LGnorm{\varphi}\leq C\WLGnorm{\varphi}.
\]
To finish the proof it suffices to show that if $u_n\to u$ in $\WLGspace$, then $\ISymbol'(u_n)\to
\ISymbol'(u)$ in $(W^1L^G)^{\ast}$.  Using  H\"older inequality and
Proposition \ref{eq:convG->convnablaG*} we obtain
\begin{multline*}
|\ISymbol'(u_{n})\varphi-\ISymbol'(u)\varphi|
=
\left|\int_I\inner{\nabla G(\dot{u}_n)-\nabla G(\dot{u})}{\dot{\varphi}} \,dt\right|
\leq 
2\LGnorm[G^{\ast}]{\nabla G(\dot{u}_n)-\nabla G(\dot{u})}\LGnorm{\dot{\varphi}}\to 0.
\end{multline*}
\end{proof}

\subsection{Case II}
We turn to general case. Suppose that $F\colon I\times\R^N\times\R^N\to \R$ satisfies
 \begin{enumerate}[label=($F_\arabic*$),series=F]
     \item
     \label{F:C1MW}
     $F\in C^1$
    \item
    \label{F:estMW}
    $ |F(t,x,v)|\leq a(|x|)(b(t)+G(v)), $
    \item
    \label{Fx:estMW}
    $|F_x(t,x,v)|\leq a(|x|)(b(t)+G(v)),$
    \item
    \label{Fv:estMW}
    $G^{\ast}(F_v(t,x,v))\leq a(|x|)(c(t)+G^{\ast}(\nabla G(v))).$
\end{enumerate}
where $a\in C(\R_+,\R_+)$, $b,c\in\LGspace[1](I,\R_+).$

If $G(v)=|v|^p $ then conditions \ref{F:estMW}, \ref{Fx:estMW} and \ref{Fv:estMW} take the standard
form (Theorem 1.4 from \cite{MawWil89}). In \cite{AciBurGiuMazSch15}  there are similar conditions
with $G(v)=\Phi(|v|)$, where $\Phi$  is an N-function. In this case, condition \ref{Fv:estMW} takes
the form $\label{Fv:est_acinasMW} |F_v(t,x,v)|\leq \tilde{a}(|x|)(\tilde{c}(t)+\Phi'(|u|)) $. In
anisotropic case we need to use $G^\ast$, because vector valued G-function is not necessarily
monotone with respect to $|\cdot|$.

Directly from \ref{Fx:estMW}, \ref{Fv:estMW} and Proposition \ref{eq:nablaGinG*} we have

\begin{lemma}
\label{lem:FxLGFvLG*}
If $u\in\WLGspace$, then $F_x(\cdot,u,\dot{u})\in\LGspace[1]$
and $F_v(\cdot,u,\dot{u})\in\LGspace[G^{\ast}]$.
\end{lemma}
\begin{proof}
Define non decreasing function
\[
\alpha(s)=\sup_{\tau\in[0,s]}a(\tau).
\]
Then, for $u\in\WLGspace$ we have
\begin{equation}
\label{eq:alpha}
a(|u(t)|)\leq \alpha(\LGnorm[\infty]{u})\leq \alpha(C\WLGnorm{u}).
\end{equation}
Let $u\in\WLGspace$. By \eqref{eq:alpha} and \ref{Fx:estMW}
\[
\int_I |F_x(t,u,\dot{u})|\,dt
\leq
\int_I a(|u(t)|)(b(t)+G(\dot{u}))\,dt
\leq
\alpha(C\WLGnorm{u})\int_I(b(t)+G(\dot{u}))\,dt<\infty.
\]
Moreover, by Proposition \ref{eq:nablaGinG*} and \ref{Fv:estMW}
\[
\int G^{\ast}(F_v(t,u,\dot{u}))\,dt\
\leq
\alpha(C\WLGnorm{u})\int_I(c(t)+G^{\ast}(\nabla G(\dot{u})))\,dt<\infty.\]
\end{proof}

\begin{theorem}
    \label{thm:IFabcC1}
    $\ISymbol\in C^1(\WLGspace,\R)$. Moreover
    \begin{equation}
    \label{IF'abc}
    \ISymbol'(u)\varphi=\int_I\inner{F_x(t,u,\dot{u})}{\varphi}dt+\int_I\inner{
        F_v(t,u,\dot{u})}{\dot{\varphi}}dt.
    \end{equation}
\end{theorem}
\begin{proof}
By \ref{F:estMW},
\begin{equation*}
|\ISymbol(u)|\leq \int_I a(|u|)(b(t)+G(\dot{u}))dt
\leq
\alpha(\WLGnorm{u})\int_I(b(t)+G(\dot{u}))dt <\infty.
\end{equation*}
It suffices to show that directional derivative $\ISymbol'(u)\in(\WLGspace)^{\ast}$ exists, is
given by \eqref{IF'abc} and that the mapping $\ISymbol':\WLGspace\to(\WLGspace)^{\ast}$  is
continuous.

Let $u\in\WLGspace$, $\varphi\in\WLGspace\setminus\{0\}$, $t\in I$, $s\in[-1,1]$. Define
\[
H(s,t):=F(t,u+s\varphi,\dot{u}+s\dot{\varphi}).
\]
By \ref{Fx:estMW}, continuity of $\varphi$, \eqref{eq:alpha} and the
fact that $u+s\varphi\in\WLGspace$ we obtain
\begin{multline*}
\int_I|\inner{F_x(t,u+s\varphi,\dot{u}+s\dot{\varphi})}{\varphi}|\,dt
\leq
\int_I|F_x(t,u+s\varphi,\dot{u}+s\dot{\varphi})||\varphi|\,dt
\leq\\ \leq
\int_Ia(|u+sv|) (b(t)+G(\dot{u}+s\dot{\varphi}))|\varphi|\,dt
\leq \\ \leq
\alpha(\WLGnorm{u+s\varphi})\int_I (b(t)+G(\dot{u}+s\dot{\varphi}))|\varphi|\,dt <\infty.
\end{multline*}
By the Fenchel inequality, \ref{Fv:estMW} and Lemma \ref{lem:FxLGFvLG*} we obtain
\[
\int_I|\inner{F_v(t,u+s\varphi,\dot{u}+s\dot{\varphi})}{\dot{\varphi}}|dt
\leq
\int_I[G^{\ast}(F_v(t,u+s\varphi,\dot{u}+s\dot{\varphi}))+G(\dot{\varphi})]dt<\infty.
\]
It follows that
\[
\int_I|H_s(s,t)|dt
=
\int_I|\inner{F_x(t,u+s\varphi,\dot{u}+s\dot{\varphi})}{\varphi}+
    \inner{F_v(t,u+s\varphi,\dot{u}+s\dot{\varphi})}{\varphi}|dt<\infty.
\]
Consequently, $\ISymbol$ has a directional derivative and
\[
\ISymbol'(u)\varphi
=
\frac{d}{ds}\ISymbol(u+s\varphi)\Big|_{s=0}
=
\int_I\inner{F_x(t,u,\dot{u})}{\varphi}dt+\int_I\inner{F_v(t,u,\dot{u})}{\dot{\varphi}}dt.
\]
By Lemma \ref{lem:FxLGFvLG*}, the H\"older inequality and \eqref{eq:Linf<WLG} we get
\[
|\ISymbol'(u)\varphi|
\leq
    \LGnorm[1]{F_x(t,u,\dot{u})}\LGnorm[\infty]{\varphi}
    +
    \LGnorm[G^{\ast}]{F_v(t,u,\dot{u})}\LGnorm{\dot {\varphi}}\leq C\WLGnorm{\varphi}.
\]

To finish the proof it suffices to show that $\ISymbol'$ is continuous. Since $u_n\to u$ in
$\WLGspace$, it follows that $u_n\to u$ in $\LGspace$, $\dot{u}_n\to \dot{u}$ in $\LGspace$ and
there exists $M>0$ such that $\WLGnorm{u_n}<M$.

By Lemma  \ref{lem:untou_implies_RGuntoRGu} we have $G(\dot{u}_n)\to G(\dot{u})$ in
$\Lpspace[1](I,\R)$. Hence
there exists a subsequence $\{u_{n_k}\}$ and $h\in \Lpspace[1](I,\R)$ such that
\[
G(\dot{u}_{n_k})\to G(\dot{u}) \text{ a.e and } G(\dot{u}_{n_k})\leq h.
\]
By \ref{Fx:estMW} and since $\{u_{n_k}\}$ is bounded, we obtain
\[
|F_x(t,u_{n_k},\dot{u}_{n_k})|
\leq
\alpha(\WLGnorm{u_{n_k}})(b(t)+G(\dot{u}_{n_k})) dt
\leq
\alpha(M)(b(t)+h(t)).
\]
By \ref{F:C1MW} we have
\[
F_x(t,u_{n_k},\dot{u}_{n_k})\to F_x(t,u,\dot{u})
\]
for a.e $t\in I$. Applying Lebesgue Dominated Convergence Theorem we obtain
\[
\int_I\inner{F_x(t,u_{n_k},\dot{u}_{n_k})}{\varphi}dt\to\int_I\inner{F_x(t,u,\dot{u})}{\varphi}dt.
\]
Since this holds for any subsequence of $\{u_n\}$ we have that
\[
\int_I\inner{F_x(t,u_{n},\dot{u}_{n})}{\varphi}dt\to\int_I\inner{F_x(t,u,\dot{u})}{\varphi}dt.
\]
By \ref{Fv:estMW} and Lemma \ref{lem:FxLGFvLG*}
\[
  G^{\ast}(F_v(t,u_{n_k},\dot{u}_{n_k}))
\leq
\alpha(M)(c(t)+G^{\ast}(\nabla G(\dot{u}_{n_k}))).
\]
In the same way as in the proof of Lemma \ref{lem:convRG*nabla} we obtain
\[
G^{\ast}(F_v(t,u_{n_k},\dot{u}_{n_k})) \leq \alpha(M)(c(t)+C+K_1h(t)).
\]
By continuity of $F_v$ we obtain
\[
G^{\ast}(F_v(t,u_{n_k},\dot{u}_{n_k}))\to G^{\ast}(F_v(t,u,\dot{u}))
\]
for a.e $t\in I$ and consequently
\[
\int_IG^{\ast}(F_v(t,u_{n_k},\dot{u}_{n_k}))dt\to\int_IG^{\ast}(F_v(t,u,\dot{u}))dt.
\]
It follows that
\[
\int_IG^{\ast}(F_v(t,u_{n},\dot{u}_{n}))dt\to\int_IG^{\ast}(F_v(t,u,\dot{u}))dt.
\]
Application of Lemma \ref{lem:rieszforLG} to $R_{G^\ast}$ yields
$\LGnorm[G^{\ast}]{F_v(\cdot,u_n,\dot{u_n})-F_v(\cdot,u,\dot{u})}\to 0$. By H\"older inequality
\[
    \Big|\int_I\inner{F_v(t,u_n,\dot{u}_n)-F_v(t,u,\dot{u})}{\dot{\varphi}}\,dt \Big|\leq
2\LGnorm[G^{\ast}]{F_v(\cdot,u_n,\dot{u}_n)-F_v(\cdot,u,\dot{u})}\LGnorm{\dot{\varphi}}\to 0.
\]
Finally,
\[
\int_I\inner{F_v(t,u_{n},\dot{u}_{n})}{\dot{\varphi}}dt
\to
\int_I\inner{F_v(t,u,\dot{u})}{\dot{\varphi}}dt.
\]

\end{proof}

\bibliographystyle{elsarticle-num}
\bibliography{OSwekt}


\end{document}